\documentclass[a4paper,12pt]{article}
\usepackage{latexsym,amsfonts,amsmath,graphics,amssymb}
\usepackage{verbatim}
\usepackage{epsfig}
\usepackage{url}

\newtheorem{theorem}{Theorem}

\newtheorem{corollary}{Corollary}

\newtheorem{algorithm}{Algorithm}
\newtheorem{remark}{Remark}

\allowdisplaybreaks

\newcommand{\abs}[1]{\left\vert#1\right\vert}

\newcommand{\bsgamma}{\boldsymbol{\gamma}}

\newcommand{\bsx}{\boldsymbol{x}}

\newcommand{\bsh}{\boldsymbol{h}}
\newcommand{\bsg}{\boldsymbol{g}}

\newcommand{\bsy}{\boldsymbol{y}}

\newcommand{\cH}{{\cal H}}

\newcommand{\icomp}{\mathtt{i}}
\newcommand{\bszero}{\boldsymbol{0}}

\newcommand{\rd}{\,\mathrm{d}}

\newcommand{\NN}{\mathbb{N}}
\newcommand{\ZZ}{\mathbb{Z}}

\newenvironment{proof}{\begin{trivlist}
    \item[\hskip\labelsep{\it Proof.}]}{$\hfill\Box$\end{trivlist}}

\allowdisplaybreaks

\begin{document}

\title{\scshape On a projection-corrected component-by-component construction}

\author{Josef Dick\thanks{The research of J. Dick was supported under Australian Research Council's 
Discovery Projects funding scheme (project number DP150101770).} 
and Peter Kritzer\thanks{P. Kritzer is supported by the Austrian Science Fund (FWF)
Project F5506-N26, which is a part of the Special Research Program "Quasi-Monte Carlo Methods: Theory and Applications".}}

\date{}
\maketitle

\begin{abstract}
The component-by-component construction is the standard method of finding good lattice rules or polynomial lattice 
rules for numerical integration. Several authors have reported that in numerical experiments the generating 
vector sometimes has repeated components. We study a variation of the classical component-by-component algorithm for the construction of 
lattice or polynomial lattice point sets where the components are forced to differ from each other. 
This avoids the problem of having projections where all quadrature points lie on the main diagonal. 
Since the previous results on the worst-case error do not apply to this modified algorithm, we prove such an 
error bound here. We also discuss further restrictions on the choice of components in the component-by-component algorithm.
 
\end{abstract}



\section*{Introduction}

Lattice point sets are integration node sets frequently used in quasi-Monte Carlo rules
$$\frac{1}{N} \sum_{n=0}^{N-1} f(\bsx_n) \approx \int_{[0,1]^s} f(\bsx) \rd \bsx$$
for the approximation of $s$-dimensional integrals over the unit cube $[0,1]^s$. 
For a modulus $N$ ($N$ a positive integer) and a generating vector $\boldsymbol{g} = (g_1, \ldots, g_s)\in \{1, 2, \ldots, N-1\}^s$, 
a (rank one) lattice point set is an integration node set of the form
\begin{equation*}
\bsx_n=\left(\left\{\frac{n g_1}{N} \right\}, \ldots, \left\{ \frac{n g_s}{N} \right\} \right), \quad n = 0, 1, \ldots, N-1.
\end{equation*}
Here, for real numbers $x \ge 0$ we write $\{x\} = x - \lfloor x \rfloor$ for the fractional part of $x$. 
For vectors $\bsx$ we apply $\{\cdot \}$ component-wise. A quasi-Monte Carlo rule using a lattice point set is called lattice rule.
For further information on lattice rules we refer to 
\cite{DKS13, N92, SJ94}. 

We consider a weighted Korobov space with general weights as studied in
\cite{DSWW06,NW10}. Before we do so we need to introduce some notation. Let
$\ZZ$ be the set of integers and let $\ZZ_{\ast}=\ZZ \setminus\{0\}$. Furthermore, $\NN$ denotes the set of positive integers. For a set $\mathcal{E}$ we denote by $|\mathcal{E}|$ the cardinality of $\mathcal{E}$.  For $s \in \NN$ we
write $[s]=\{1,2,\ldots,s\}$. For a vector $\bsx=(x_1,\ldots,x_s)\in [0,1]^s$
and for $u \subseteq [s]$ we write $\bsx_u=(x_j)_{j \in u} \in
[0,1]^{|u|}$ and $(\bsx_{u},\bszero)\in [0,1]^s$ for the vector
$(y_1,\ldots,y_s)$ with $y_j=x_j$ if $j \in u$ and $y_j=0$ if $j \not\in
u$. 

The importance of the different components or groups of components of the functions from 
the Korobov space to be defined is specified by a sequence of positive weights 
$\bsgamma=(\gamma_{u})_{u \subseteq [s]}$, see \cite{SW98}, where we may assume that $\gamma_{\emptyset}=1$. 
The smoothness is described by a parameter $\alpha>1$. The weighted Korobov space $\cH(K_{s,\alpha,\bsgamma})$ 
is a reproducing kernel Hilbert space with kernel function of the form 
$$
K_{s,\alpha,\bsgamma}(\bsx,\bsy) = 
1+ \sum_{\emptyset \not=u \subseteq [s]} 
\gamma_{u} \sum_{\bsh_{u}\in \ZZ_{\ast}^{|u|}} \frac{\exp(2 \pi \icomp \bsh_{u}\cdot (\bsx_{u}-\bsy_{u}))}{\prod_{j \in u}|h_j|^{\alpha}}.
$$
It is well known in the theory of lattice rules that it is useful to restrict the range of a generating vector $\bsg$ 
of an $N$-point lattice point set to $\mathcal{Z}_N^s$, where 
\begin{equation*}
\mathcal{Z}_{N} = \{ k \in \{1, 2, \ldots, N-1\}: \gcd(k, N) = 1\}.
\end{equation*}
Furthermore, it is known (see, for example, \cite{DSWW06}) that the squared worst-case 
error of a lattice rule generated by a generating vector $\bsg \in \mathcal{Z}_N^s$ 
in the weighted Korobov space $\cH(K_{s,\alpha,\bsgamma})$ is given by 
\begin{equation}\label{eqerrorexpr}
e^2(\boldsymbol{g}) = \sum_{\substack{ \boldsymbol{h} \in \mathbb{Z}^s \setminus \{ \boldsymbol{0}\} 
\\ \boldsymbol{g} \cdot \boldsymbol{h} \equiv 0 \pmod{N} } } r_\alpha(\boldsymbol{\gamma}, \boldsymbol{h} ),
\end{equation}
where for $\emptyset\neq u\subseteq [s]$ and $\boldsymbol{h}_{u} \in \mathbb{Z}_\ast^{\abs{u}}$ we have
\begin{equation*}
r(\boldsymbol{\gamma}, (\boldsymbol{h}_u, \boldsymbol{0}) ) = \gamma_u \prod_{j \in u} |h_j|^{-\alpha} .
\end{equation*}

It is known that the worst-case error in the Korobov space coincides with the worst-case error in the unanchored Sobolev 
space using the tent-transform \cite{DNP14} and is also related to the mean square worst-case error for randomly shifted lattice rules 
\cite{DKS13}. Hence the results here automatically also apply to those cases.

\section*{The result}

The now standard method for finding good generating vectors for numerical integration in Korobov spaces is the so-called component-by-component 
(CBC) construction (see~\cite{D04,K03}). We can set the first component to $1$ and then proceed inductively by 
choosing one new component at a time by minimizing the error criterion $e^2(g_1^\ast, g_2^\ast,\ldots, g_{d-1}^\ast, g)$ 
as a function of the last (not yet fixed) component $g\in \mathcal{Z}_N$. 
Here, the components $g_1^\ast, g_2^\ast, \ldots, g_{d-1}^\ast$ have been fixed in the previous steps.

It has been observed that in running the CBC construction it may happen 
that components repeat themselves, i.e., there are $i, j \in \{1, \ldots, s\}$ such that $g_i^\ast = g_j^\ast$. We quote from \cite{FK}:
\begin{quote}
\textit{[\ldots] However, it has been observed that the components start to repeat from some dimension onward 
for product-type weights, hence leading to a practical limit on the value of $d$ [we remark that $d$ has the role of $s$ in \cite{FK}]. 
This side effect of the CBC algorithm is yet to be fully understood.}
\end{quote}
This problem may be due to numerical issues of the CBC algorithm, see \cite[p. 386]{NC06}, but this is currently not known. This paper also does not contribute to an understanding of this problem, instead we study a method to avoid its occurrence. Another quote is from \cite{GS15}, where it is stated that:
\begin{quote}
\textit{[\ldots] For large values of the worst-case error, the elements of the generating vector can repeat, 
leading to very bad projections in certain dimensions. }
\end{quote}
To alleviate this problem, Gantner and Schwab~\cite{GS15} introduce a method they call pruning in the CBC algorithm. 
If $g_1^\ast, \ldots, g_{d-1}^\ast$ have already been chosen by the CBC algorithm, 
then they choose the $d$th component from the set $\mathcal{Z}_N \setminus \{g^\ast_1, \ldots, g^\ast_{d-1}\}$, 
which forces the new component to differ from all the previous components.

Following this idea, we study a modified CBC algorithm which excludes all 
values in a set $\mathcal{E}_d \subsetneqq \mathcal{Z}_N$ when choosing the $d$th component. Note that we allow the sets $\mathcal{E}_d$ to depend on the values of $g_1^\ast,\ldots,g_{d-1}^\ast$ for $d\in\{2,\ldots,s\}$.
(As just mentioned, \cite{GS15} considered the special case $\mathcal{E}_d = \{g_1^\ast, \ldots, g_{d-1}^\ast\}$.) 
The standard CBC algorithm can be obtained by setting $\mathcal{E}_2 = \ldots = \mathcal{E}_s = \emptyset$, 
i.e. no components are excluded in the CBC construction. We discuss other sets of exclusions in Section~\ref{sec_soe}.

In the following, we write $\phi$ for Euler's totient function.
\begin{algorithm}\label{alg_lat}
Let $N, s  \in \mathbb{N}$ be given. 
\begin{itemize}
\item[(i)] Set $g^\ast_1 = 1$, and choose $\mathcal{E}_2 \subsetneqq \mathcal{Z}_N$. If no components are to be excluded in coordinate 2, then 
set $\mathcal{E}_2 = \emptyset$.
\item[(ii)] For $d=1, \ldots, s -1$, do the following: assume that $g_1^\ast, \ldots, g_{d}^\ast$ have 
already been found, and choose $\mathcal{E}_{d+1} \subsetneqq \mathcal{Z}_N$. If no components are to be excluded in coordinate $d+1$, then set $\mathcal{E}_{d+1} = \emptyset$.  Find $g_{d+1}^\ast$ as the minimizer $g\in \mathcal{Z}_{N} \setminus \mathcal{E}_{d+1}$ of
\begin{equation*}
e^2(g_1^\ast, \ldots, g_{d}^\ast, g).
\end{equation*}
\end{itemize}
\end{algorithm}
In dimension $d=1$ we do not use any exclusions since in our setting all components yield the same point set.

One can still use the fast CBC method of \cite{NC06b, NC06c} for this approach, with the additional step of checking whether a component is in the set of exclusions (as pointed out in the case of repeated components in \cite[Section~4.2]{GS15}). That is, in component $d$ one needs to perform at most $|\mathcal{E}_d|$ checks for exclusions, hence one needs to perform at most an additional $|\mathcal{E}_2| + \cdots + |\mathcal{E}_{s}| \le (s - 1) (\phi(N)-1) \le s N$ checks. 
This does not increase the overall complexity of the fast CBC algorithm.

Using exclusions in the component-by-component construction (for instance, by forcing new components to differ from the previous ones) implies that the theoretical results on the component-by-component construction as shown, for instance, in \cite{D04,DPW08,K03} do not apply anymore. Hence it has remained an open question as to what theoretical bounds one can get in this case. The following theorem provides an answer to this question. 

For simplicity we assume in the following theorem that the weights are of the form $\gamma_u = \prod_{j \in u} \widetilde{\gamma}_j$. However, it is clear that the result holds for any set of nonnegative numbers $(\gamma_u)_{u \subseteq [s]}$. We write $\zeta(\alpha) = \sum_{h=1}^\infty h^{-\alpha}$ for the Riemann zeta function.

\begin{theorem}\label{thm_main}
Let $N, s \in \mathbb{N}$ be given. Let $\gamma_u = \prod_{j \in u} \widetilde{\gamma}_j$ for nonnegative real numbers $\widetilde{\gamma}_j$. 
Assume that $\boldsymbol{g}^\ast = (g_1^\ast, \ldots, g_s^\ast)$ and sets of exclusions  
$\mathcal{E}_{2},\ldots,\mathcal{E}_{s} \subsetneqq \mathcal{Z}_N$
have been constructed by the algorithm. Then for all $1 \le d \le s$ we have
\begin{equation}\label{eq_main}
e^2(g_1^\ast,\ldots, g_d^\ast) \le \left(\frac{1}{\phi(N)} \sum_{u \subseteq [d] } 
\gamma_u^\lambda (2 \zeta(\alpha \lambda))^{|u|} \prod_{j \in u } \frac{\phi(N)}{\phi(N)- |\mathcal{E}_j|} \right)^{1/\lambda},
\end{equation}
for any $1/\alpha < \lambda \le 1$, where the product over the empty set is defined as $1$.
\end{theorem}

\begin{proof}
We prove the result by induction on $d$. For $d=1$, let $1/\alpha < \lambda \le 1$. Then the result holds since
\begin{equation*}
e^2(g_1^\ast) = \gamma_{\{1\}} \sum_{h \in \mathbb{Z}\setminus \{0\}} |N h|^{-\alpha } = 
\frac{\gamma_{\{1\}}}{N^\alpha } 2 \zeta(\alpha) \le \left(\frac{\gamma_{\{1\}}^\lambda }{\phi(N)} 2 \zeta(\alpha \lambda) \right)^{1/\lambda}.
\end{equation*}

Assume that $\bsg^\ast = (g_1^\ast, \ldots, g_{d}^\ast)$ is chosen according to the algorithm and 
that \eqref{eq_main} holds for $d$ for any choice of $1/\alpha < \lambda \le 1$.

From \eqref{eqerrorexpr} it is easy to deduce that, for $g\in \mathcal{Z}_N$, 
\begin{equation}\label{eqindstep} 
e^2(g_1^\ast,\ldots, g_{d}^\ast, g)=e^2(g_1^\ast, \ldots, g_{d}^\ast)+\theta_{N, d+1, \alpha, \boldsymbol{\gamma} }(g),
\end{equation}
where 
\begin{align}
\theta_{N, d+1, \alpha, \boldsymbol{\gamma} }(g)= & 
\sum_{\substack{ \bsh \in \mathbb{Z}^{d+1},\ h_{d+1} \neq 0 \\  \bsh \cdot (\bsg^\ast, g) \equiv 0 \pmod{N} } } r_\alpha(\bsgamma, \bsh). \label{eq_theta}
\end{align}

From \cite[Eq.~(5)]{D04} (setting $\beta_j = 1$) we obtain
\begin{equation*}
\theta_{N, d+1, \alpha, \boldsymbol{\gamma} }(g)=2 \widetilde{\gamma}_{d+1} \zeta (\alpha) N^{-\alpha} (1 + e^2(g_1^\ast,\ldots, g_d^\ast) ) 
+ \widetilde{\gamma}_{d+1} \kappa_{N,d+1\alpha,\bsgamma}(g),
\end{equation*}
where
\begin{equation}\label{eqkappa}
\kappa_{N,d+1,\alpha,\bsgamma}(g) = \sum_{\substack{ h_{d+1} \in \mathbb{Z} \\ N \nmid h_{d+1} }}\ \ \ \sum_{\substack{ \bsh \in \mathbb{Z}^d \\ 
\bsh \cdot \bsg^\ast \equiv - h_{d+1} g \pmod{N} }} |h_{d+1}|^{-\alpha} r_\alpha(\bsgamma, \bsh).
\end{equation}
Hence we obtain from \eqref{eqindstep}
\begin{equation}\label{eqindstep2} 
e^2(g_1^\ast,\ldots, g_{d}^\ast, g)=(1+2 \widetilde{\gamma}_{d+1} \zeta (\alpha) N^{-\alpha})e^2(g_1^\ast, \ldots, g_{d}^\ast)+ 
2 \widetilde{\gamma}_{d+1} \zeta (\alpha) N^{-\alpha}+\widetilde{\gamma}_{d+1} \kappa_{N,d+1,\alpha,\bsgamma}(g).
\end{equation}

Let now $\lambda^\ast \in(1/\alpha,1]$ be chosen such that the right hand side of \eqref{eq_main} for $d+1$ is minimal for $\lambda^\ast$. 
In the following, we write $ \boldsymbol{\gamma}^{\lambda^*}$ for the weights $\gamma_u^{\lambda^*}=\prod_{j\in u}\widetilde{\gamma}_j^{\lambda^*}$.
We apply Jensen's inequality $( \sum_k a_k )^{\lambda^\ast} \le \sum_k a_k^{\lambda^\ast}$ to \eqref{eqindstep2} to obtain
$$
(e^2(g_1^\ast,\ldots, g_{d}^\ast, g))^{\lambda^*}\le
(1+2^{\lambda^*} \widetilde{\gamma}_{d+1}^{\lambda^*} \zeta (\alpha\lambda^*) N^{-\alpha\lambda^*})
(e^2(g_1^\ast, \ldots, g_{d}^\ast))^{\lambda^*} 
+ 2^{\lambda^*} \widetilde{\gamma}_{d+1}^{\lambda^*} \zeta (\alpha\lambda^*) N^{-\alpha\lambda^*}
+ \widetilde{\gamma}_{d+1}^{\lambda^*} ( \kappa_{N,d+1,\alpha,\bsgamma}(g))^{\lambda^*}.
$$ 
Applying Jensen's inequality to \eqref{eqkappa}, we easily see that
\begin{equation*}
\frac{1}{\phi(N)}  
\sum_{\ell \in \mathcal{Z}_N} (\kappa_{N, d+1, \alpha, \boldsymbol{\gamma} }(\ell))^{\lambda^\ast} 
\le \frac{1}{\phi(N)} \sum_{\ell \in \mathcal{Z}_N} \kappa_{N, d+1, \alpha\lambda^*, \boldsymbol{\gamma}^{\lambda^*} }(\ell)
=:\overline{\kappa}_{N, d+1, \alpha\lambda^*, \boldsymbol{\gamma}^{\lambda^*} }
\end{equation*}
In the following we use ideas similar to \cite{DPW08, DSWW06}. We now use Markov's inequality, 
which states that for a nonnegative random variable $X$ and any real number $c \ge 1$ we have that 
$\mathbb{P}(X < c \mathbb{E}(X)) > 1 -  c^{-1}$. We use the normalized counting measure $\mu$ on $\mathcal{Z}_N$ as the probability measure. 
For $c \ge 1$ let
\begin{align*}
G_c := & \left\{g \in \mathcal{Z}_N: (\kappa_{N, d+1, \alpha, \boldsymbol{\gamma} }(g))^{\lambda^\ast} 
\le c \overline{\kappa}_{N, d+1, \alpha \lambda^\ast, \boldsymbol{\gamma}^{\lambda^\ast} }  \right\} \\ 
\supseteq & \left\{g \in \mathcal{Z}_N: (\kappa_{N, d+1, \alpha, \boldsymbol{\gamma} }(g))^{\lambda^\ast} \le \frac{c}{\phi(N)}  
\sum_{\ell \in \mathcal{Z}_N} (\kappa_{N, d+1, \alpha, \boldsymbol{\gamma} }(\ell))^{\lambda^\ast} \right\} =: A_c.
\end{align*}
Then
\begin{equation*}
\mu(G_c) = \frac{|G_c|}{\phi(N)} \ge \mu(A_c) = \frac{|A_c|}{\phi(N)} > 1 - \frac{1}{c}.
\end{equation*}
In other words, for any $c \ge 1$, there is a subset $G_c \subseteq \mathcal{Z}_N$ of size bigger than 
$\phi(N) (1 - c^{-1})$ such that
\begin{equation*}
(\kappa_{N, d+1, \alpha, \boldsymbol{\gamma} }(g))^{\lambda^\ast} \le  c\overline{\kappa}_{N, d+1, 
\alpha \lambda^\ast, \boldsymbol{\gamma}^{\lambda^\ast}},\ \ \forall g\in G_c.
\end{equation*}
By choosing $c \ge 1$ such that
\begin{equation*}
\phi(N) \left(1-\frac{1}{c} \right) = |\mathcal{E}_d|,
\end{equation*}
it follows that the set $G_c \setminus \mathcal{E}_d$ is not empty. This condition is satisfied for
\begin{equation*}
c = \frac{\phi(N)}{\phi(N) - |\mathcal{E}_d| }.
\end{equation*}
In particular, if $\mathcal{E}_d = \emptyset$ then $c = 1$. As $g_{d+1}^\ast$ is chosen by the algorithm 
such that the error $e^2(g_1^\ast,\ldots, g_{d}^\ast, g)$ is minimized, we obtain 
\begin{eqnarray}\label{eqindstep3} 
(e^2(g_1^\ast,\ldots, g_{d}^\ast, g_{d+1}^\ast))^{\lambda^*}&\le& (1+2^{\lambda^*} \widetilde{\gamma}_{d+1}^{\lambda^*} \zeta (\alpha\lambda^*) N^{-\alpha\lambda^*})
(e^2(g_1^\ast, \ldots, g_{d}^\ast))^{\lambda^*}+ 2^{\lambda^*} \widetilde{\gamma}_{d+1}^{\lambda^*} \zeta (\alpha\lambda^*) N^{-\alpha\lambda^*} +
c \widetilde{\gamma}_{d+1}^{\lambda^*}\overline{\kappa}_{N, d+1, 
\alpha \lambda^\ast, \boldsymbol{\gamma}^{\lambda^\ast}}\nonumber\\
&\le& (1+c2^{\lambda^*} \widetilde{\gamma}_{d+1}^{\lambda^*} \zeta (\alpha\lambda^*) N^{-\alpha\lambda^*})
(e^2(g_1^\ast, \ldots, g_{d}^\ast))^{\lambda^*}+c2^{\lambda^*} \widetilde{\gamma}_{d+1}^{\lambda^*} \zeta (\alpha\lambda^*) N^{-\alpha\lambda^*} +
c \widetilde{\gamma}_{d+1}^{\lambda^*}\overline{\kappa}_{N, d+1, 
\alpha \lambda^\ast, \boldsymbol{\gamma}^{\lambda^\ast}}\nonumber\\
&\le& (1+c2 \widetilde{\gamma}_{d+1}^{\lambda^*} \zeta (\alpha\lambda^*) N^{-\alpha\lambda^*})
(e^2(g_1^\ast, \ldots, g_{d}^\ast))^{\lambda^*}+c2\widetilde{\gamma}_{d+1}^{\lambda^*} \zeta (\alpha\lambda^*) N^{-\alpha\lambda^*} +
c \widetilde{\gamma}_{d+1}^{\lambda^*}\overline{\kappa}_{N, d+1, 
\alpha \lambda^\ast, \boldsymbol{\gamma}^{\lambda^\ast}}. \nonumber\\
\end{eqnarray}
Using the induction assumption with $\lambda=\lambda^\ast$, we obtain
\begin{equation}\label{eq_ind_assump}
 (e^2(g_1^\ast, \ldots, g_{d}^\ast))^{\lambda^*}\le 
\frac{1}{\phi(N)} \sum_{u \subseteq [d] } 
\gamma_u^{\lambda^*} (2 \zeta(\alpha \lambda^*))^{|u|} \prod_{j \in u} \frac{\phi(N)}{\phi(N)- |\mathcal{E}_j|}.
\end{equation}
Furthermore, from the proof of \cite[Lemma~5]{D04} we obtain
\begin{eqnarray}\label{eqkappaest}
\overline{\kappa}_{N, d+1, \alpha\lambda^*, \boldsymbol{\gamma}^{\lambda^*} } &\le& 2 \zeta(\alpha\lambda^*) (1-N^{-\alpha\lambda^*}) \phi(N)^{-1} 
\sum_{\emptyset \neq u \subseteq [d] } \gamma_u^{\lambda^*} (2 \zeta(\alpha\lambda^*) )^{|u|}\nonumber\\
&\le&2 \zeta(\alpha\lambda^*) (1-N^{-\alpha\lambda^*}) \phi(N)^{-1} 
\sum_{\emptyset \neq u \subseteq [d] } \gamma_u^{\lambda^*} (2 \zeta(\alpha\lambda^*) )^{|u|}
\prod_{j \in u } \frac{\phi(N)}{\phi(N)- |\mathcal{E}_j|}.
\end{eqnarray}
Inserting the bounds in \eqref{eq_ind_assump} and \eqref{eqkappaest} into \eqref{eqindstep3}, and noting 
that $N^{-\alpha\lambda^*}\le \phi (N)^{-1}$, we obtain
\begin{eqnarray*}\label{eqerrestimate} 
(e^2(g_1^\ast,\ldots, g_{d}^\ast, g_{d+1}^\ast))^{\lambda^*}&\le&
\frac{1}{\phi(N)} \sum_{u \subseteq [d] } 
\gamma_u^{\lambda^*} (2 \zeta(\alpha \lambda^*))^{|u|} \prod_{j \in u } \frac{\phi(N)}{\phi(N)- |\mathcal{E}_j|}\\
&&+ c2\widetilde{\gamma}_{d+1}^{\lambda^*} \zeta (\alpha\lambda^*) \frac{1}{\phi(N)}\\
&&+ c2\widetilde{\gamma}_{d+1}^{\lambda^*} \zeta (\alpha\lambda^*) \frac{1}{\phi(N)}
\sum_{\emptyset \neq u \subseteq [d] } \gamma_u^{\lambda^*} (2 \zeta(\alpha\lambda^*) )^{|u|}
\prod_{j \in u } \frac{\phi(N)}{\phi(N)- |\mathcal{E}_j|} \\  & \le & \frac{1}{\phi(N)} \sum_{u \subseteq [d+1] } \gamma_u^{\lambda^*} (2 \zeta(\alpha \lambda^*))^{|u|} \prod_{j \in u } \frac{\phi(N)}{\phi(N)- |\mathcal{E}_j|}.
\end{eqnarray*}
This implies the desired error bound for the special case of $\lambda^*$. However, since we chose $\lambda^*$ such that 
the right hand side of \eqref{eq_main} is minimal, the result also holds for arbitrary $\lambda\in (1/\alpha,1]$.
\end{proof}

The following corollary considers the case where the relative size of the set of exclusions is uniformly bounded. 

\begin{corollary}
Let sequences of positive integers $(N_k)_{k\in \mathbb{N}}$ and $(s_k)_{k \in \mathbb{N}}$ 
be given. Let $\boldsymbol{g}^\ast_k = (g_{1,k}^\ast, \ldots, g_{s_k, k}^\ast)$ be constructed by the 
algorithm using $N_k, s_k$ and the set of exclusions $\mathcal{E}_{2}^{(k)}, \ldots, \mathcal{E}_{s_k}^{(k)}$ for $k \in \mathbb{N}$. 
Assume that there is a $0 < \delta < 1$ such that 
\begin{equation}\label{unif_bounded}
\sup_{k \in \mathbb{N}} \frac{ \max_{2 \le d \le s_k} |\mathcal{E}_{d}^{(k)}| }{\phi(N_k)}  \le \delta.
\end{equation}

Then
\begin{equation*}
e^2(\boldsymbol{g}^\ast_k) \le \left(\frac{1}{\phi(N_k)} \sum_{u \subseteq [s_k] } 
\gamma_u^\lambda \left( \frac{2 \zeta(\alpha \lambda)}{1-\delta} \right)^{|u|}  \right)^{1/\lambda},
\end{equation*}
for any $1/\alpha < \lambda \le 1$ and any $k \in \mathbb{N}$.
\end{corollary}

The corollary illustrates that as long as the relative size of the sets of exclusions is uniformly bounded \eqref{unif_bounded},  tractability results are not effected. In other words, if one, for instance, gets strong polynomial QMC tractability using the standard CBC algorithm, one also gets strong polynomial QMC tractability for the modified CBC algorithm using uniformly bounded sets of exclusions.

\begin{remark}{\ }

\begin{itemize}
\item[(i)] Similar results to the theorem and the corollary hold if one considers polynomial lattice rules (cf.~\cite{N92}) instead of lattice rules (using an approach similar to \cite{D07} instead of \cite{D04}).
\item[(ii)] Our method can also be applied to interlaced polynomial lattice rules \cite{G, GD15}. If the aim is to have different components in each coordinate, then due to the interlacing of consecutive coordinates in blocks of length $d$, it is enough to only force different first components of each block. This makes the additional construction cost and the increase in the error bound independent of the interlacing factor $d$.
\end{itemize}
\end{remark}

\section*{Some particular choices for sets of exclusions}\label{sec_soe}

Although the algorithm, the theorem and the corollary apply to arbitrary sets of exclusions, some particular choices are of general interest. We discuss some of them in the following.

\subsection*{Repeated components}

If the aim is simply to avoid repeated components as observed in some numerical experiments, one can simply choose the sets of exclusions
\begin{equation*}
\mathcal{E}_d = \{g_1^\ast, \ldots, g_{d-1}^\ast\}, \quad d = 2, \ldots, s.
\end{equation*}
This makes sure that there are no two-dimensional projections of the integration lattice whose points all lie on the main diagonal $\{(x, x): 0 \le x \le 1\}$.

\subsection*{Avoiding diagonals}

To also exclude having two-dimensional projections where all the points of the integration lattice lie on an antidiagonal $\{(x, 1-x): 0 \le x \le 1\}$, one can additionally exclude the components $N- g_1^\ast, N- g_2^\ast, \ldots, N-g_{d-1}^\ast$. This suggests to use the sets of exclusions
\begin{equation}\label{avoid_diag}
\mathcal{E}_d = \{g_1^\ast, N-g_1^\ast, g_2^\ast, N-g_2^\ast, \ldots, g_{d-1}^\ast, N-g_{d-1}^\ast\}, \quad d = 2, \ldots, s.
\end{equation}
Note that this is only possible as long as $2(s-1) < \phi(N)$. Even so, if, say $|\mathcal{E}_s| = \phi(N) - \ell$ for some small $\ell \in \mathbb{N}$, then the factor $\phi(N)/(\phi(N) - |\mathcal{E}_s|) = \phi(N) / \ell$ becomes large, 
in which case the bound in the theorem becomes meaningless. So one still wants to impose a restriction of the form, say, $\max_{1 \le d \le s} |\mathcal{E}_d| \le \delta \phi(N)$ for some `reasonable choice' (depending on the application) of $\delta < 1$. The last inequality implies that $2 (s-1) \le \delta \phi(N)$.

\subsection*{Avoiding diagonals in smaller dimension}

In some circumstances a condition of the form $2(s-1) \le \delta \phi(N)$ cannot be satisfied. 
For instance, when considering tractability questions one wants to study the dependence on the dimension as $s$ tends to $\infty$ \cite{NW10}. 
Another case in which problems can arise is when $N$ and $s$ need to be increased simultaneously, as for instance in \cite{KSS12} and \cite{DKLNS14}. 
In this case one can, for instance, choose $\mathcal{E}_d$ as in \eqref{avoid_diag} for $d=1, \ldots, s^\ast$ for some fixed $s^\ast$ 
(independent of $N$ and $s$) and set $\mathcal{E}_d = \emptyset$ for $d = s^\ast+1, \ldots, s$. As long as $2(s^\ast-1) \le \delta \phi(N)$, 
the corollary applies since the relative size of the sets of exclusions is uniformly bounded and therefore strong polynomial QMC tractability results can be obtained. The particular choice of $s^\ast$ will depend on the problem under consideration.

\subsection*{Reduced fast CBC}

Another instance where a particular type of sets of exclusions has been considered is \cite{DKLP15}. In this case the aim was somewhat different, namely, to reduce the search space in coordinate $d$ such that one obtains a speed-up of the fast CBC algorithm. In this case, instead of having to do additional computational work, the computational work actually decreases. See \cite{DKLP15} for details.

\bigskip

\noindent {\bf Acknowledgments.} P. Kritzer would like to thank J. Dick and F.Y. Kuo for their hospitality during 
his stay at the University of New South Wales, where this paper was written. The research of J. Dick was supported under Australian Research Council's Discovery Projects funding scheme (project number DP150101770). P. Kritzer is supported by the Austrian Science Fund (FWF) Project F5506-N26, which is a part of the Special Research Program "Quasi-Monte Carlo Methods: Theory and Applications". The authors are very grateful to the reviewers for many helpful comments, in particular, for asking for general sets of exclusions.

\noindent {\bf Authors' addresses:} \\

Josef Dick, School of Mathematics and Statistics, The University of New South Wales, Sydney, 2052 NSW, Australia. e-mail: josef.dick(AT)unsw.edu.au \\

Peter Kritzer, Department of Financial Mathematics and Applied Number Theory, Johannes Kepler University Linz, Altenbergerstr. 69, 
4040 Linz, Austria. e-mail: peter.kritzer(AT)jku.at

\end{document}